\documentclass[11pt,a4paper]{article}
\usepackage{amsmath,mathbbol,amssymb,mathbbol}
\usepackage[T1]{fontenc} \usepackage{mathpazo}
\usepackage[active]{srcltx}

\usepackage{amsmath}
\usepackage{amssymb}
\usepackage{amsbsy}
\usepackage{amsfonts}
\usepackage{graphicx,subfigure,alphalph}
\usepackage{verbatim}
\usepackage{enumitem}
\usepackage{hyperref}
\usepackage{ifthen}
\usepackage{tikz}
\usepackage{psfrag}
\usepackage{epsfig}
\usepackage{color}

 \graphicspath{{figures/}}
 \DeclareGraphicsExtensions{.jpg,.eps,.ps,.pdf,.png}

\makeatletter

\newcommand{\ve}{\varepsilon}
\def\d{\partial}
\def\O{\Omega}
\def\E{\mathcal E}

\def\R{{\mathbb R}}

 \def\E{{\mathcal E}}

\DeclareMathOperator{\dv}{div}
\DeclareMathOperator{\Dom}{dom} 

\newcommand{\dom}[1]{\Dom {#1}}  

\newtheorem{theorem}{Theorem}[section]
\newtheorem{lemma}[theorem]{Lemma}

\newenvironment{proof}{{\textbf{Proof.}}}{\hfill \textbf{$\square$}\vspace{0.2cm}}

\title{The bidomain problem as a gradient system}

\author{Zakaria Belhachmi
\thanks{Laboratoire de Math\'ematiques LMIA, Universit\'e de Haute Alsace, 4, rue des Fr\`eres Lumi\`ere,
68096 Mulhouse, FRANCE.(\tt{zakaria.belhachmi@uha.fr})} \and
Ralph Chill
\thanks{
Institut f\"ur Analysis, Fakult\"at Mathematik, TU
Dresden, 01062 Dresden, Germany.(\tt{ralph.chill@tu-dresden.de})}
}

\begin{document}

\maketitle

%\date{\today}

%\maketitle

\begin{abstract}
We consider a general, nonlinear version of the bidomain system. Using the gradient structure of this system, but also the notion of $j$-subgradient, we prove wellposedness of the bidomain system in the energy space and provide first numerical experiments.
\end{abstract}

\section{The bidomain model}

We consider the following nonlinear version of the bidomain model arising in Hodgkin \& Huxley \cite{HoHu52} and Neu \& Krassowska \cite{NeKr93}, 
\begin{align}
\label{b1} \partial_t (u_i - u_e) - \dv q_i (x,  \nabla u_i ) + \frac{\partial F}{\partial u} (u_i-u_e,w) & = 0 && \text{in } (0,\infty )\times \Omega ,\\
\label{b2} \dv q_i (x, \nabla u_i ) + \dv q_e (x, \nabla u_e ) & = 0 && \text{in } (0,\infty )\times \hat{\Omega} , \\
% \dv q_e (x, \nabla u_e ) & = 0 && \text{in } (0,\infty )\times \hat{\Omega}\setminus \Omega , \\
% u_i - u_e & = v  && \text{in } \Omega , \\
\label{b3} q_i (x, \nabla u_i) \cdot n & = 0 && \text{on } (0,\infty )\times \partial\Omega , \\
% [u_e] = [q_e (x, \nabla u_e) \cdot n] & = 0 && \text{on } (0,\infty )\times \partial\Omega , \\
\label{b4} q_e (x, \nabla u_e) \cdot n & = 0 && \text{on } (0,\infty )\times \partial\hat{\Omega} ,\\
\label{b4a} \tau\, \partial_t w + \frac{\partial F}{\partial w} (u_i-u_e,w) & = 0 && \text{in } (0,\infty )\times\Omega , \\
\label{b5} u_i (0,\cdot ) - u_e (0,\cdot ) & = u_0  && \text{in } \Omega , \\
\label{b6} w(0,\cdot ) & = w_0  && \text{in }\Omega .
\end{align}
%\sideremark{Where is $I_{ion}$?}
Here, $\Omega$, $\hat{\Omega}\subseteq\R^N$ are two bounded domains with $C^1$-boundaries and such that $\Omega \subseteq \hat{\Omega}$, equality being possible, and $\tau >0$ is a real constant. For the coefficients $q_i : \Omega \times \R^N \to\R^N$ and $q_e : \hat{\Omega} \times \R^N\to\R^N$ we assume that they are gradients in the second variables, that is,
\[
 q_i (x,y) = \nabla_y Q_i (x,y) \text{ and } q_e (x,y) = \nabla_y Q_e (x,y) 
\]
for two functions $Q_i : \Omega \times \R^N \to\R$ and $Q_e:\hat{\Omega}\times\R^N\to\R$ satisfying the Caratheodory conditions
\begin{align}
\label{q1} & Q_i \text{ and } Q_e \text{ are measurable in the first variable and} \\
 \nonumber & \text{differentiable and strictly convex in the second variable,} \\
\label{q2}  & Q_i (x,0) = 0 \text{ and } Q_e (x,0) = 0 , \\ %\sidemark{Second condition only normalisation}
\label{q3} & Q_i (x,y) , \, Q_e (x,y) \geq \alpha \, |y|^p \text{ for some } \alpha > 0 \text{ and every } \\
\nonumber & x\in \Omega \text{ (resp. } x\in\hat{\Omega}\text{) and every } \, y\in\R^N . 
%\label{q4} & |q_i (x,y)| , \, |q_e (x,y)| \leq C\, |y|^{p-1} \text{ for some } C \geq 0 \text{ and every } x\in \Omega , \, y\in\R^N . 
\end{align}
%\sideremark{Which conditions precisely?}
Here, $p\in ]1,\infty [$ is fixed. Moreover, we assume that 
\begin{equation} \label{f}
\begin{split}
& F\in C^1 (\R^2 ) \text{ is semiconvex in the sense that for some } \omega\in\R  \\
& \text{the function } (u,w) \mapsto F(u,w) +\frac{\omega}{2} (u^2+w^2) \text{ is convex.} 
%\text{ is such that the Hessian } H_F \text{ is bounded from below.}
\end{split}
\end{equation}
%\sideremark{Alternatively: $F\in W^{2,\infty}$}
The equations \eqref{b1}, \eqref{b2} and \eqref{b5} on the two domains $\Omega$ and $\hat{\Omega}$ are to be understood in the following way: if the function $u_e$ is considered on the smaller domain $\Omega$ (like in equations \eqref{b1} or \eqref{b5}), then we mean the restriction of $u_e$ to this domain, and if a function is a priori only given on $\Omega$, like for example the function $\dv q_i (x,\nabla u_i)$ in equation \eqref{b2}, then we extend it by $0$ to the larger domain $\hat{\Omega}$. 

Note that we impose no growth restrictions on the functions $Q_i$, $Q_e$ and $F$ from above. %, and the energy defined below may actually take also the value $+\infty$. 
The sequel shows that one may actually allow more general conditions  without essentially changing the results. For example, the exponent $p$ in the growth conditions on $Q_i$ and $Q_e$ need not be the same; one may allow two different exponents $p_i$, $p_e\in ]1,\infty [$. One may also consider growth conditions involving Young functions other than the $p$-powers, or growth conditions which depend on $x\in\Omega / \hat{\Omega}$; this would lead to other energy spaces, involving namely Orlicz spaces, variable $L^p$-spaces or more general spaces of these types. We do not go into details here. 

The bidomain model has first been mathematically analysed in Ambrosio, Colli Franzone \& Savar\'e \cite{AmCoSa00} and Colli Franzone \& Savar\'e \cite{CoSa02}. In \cite{AmCoSa00}, associated energy functionals (in much higher generality than considered below) and their $\Gamma$-convergence have been studied; see also Colli Franzone, Pavarino \& Savar\'e \cite{CoPaSa06} and Colli Franzone, Pavarino \& Scacchi \cite{CoPaSc14}. In a particular semilinear case, that is, the elliptic operators in the system \eqref{b1}--\eqref{b6} are linear, it has been remarked that the system above has a gradient structure in the sense that the associated energy functionals decrease along solutions. In accordance with this observation, wellposedness in a Hilbert space setting and in the semilinear case has been proved in \cite{CoSa02} by variational methods which look very similar to the method of $j$-elliptic, bilinear forms developed recently by Arendt \& ter Elst \cite{ArEl12,ArEl12a}; compare also with Veneroni \cite{Ve06,Ve09}. Bourgault, Coudi\`ere \& Pierre \cite{BoCoPi09} proved wellposedness by reducing the degenerate system %(actually, a coupled parabolic--elliptic system) 
to an abstract semilinear Cauchy problem in which the leading linear operator is obtained as a ``harmonic mean'' of the two elliptic operators appearing in equation \eqref{b2}; see also Giga \& Kajiwara \cite{GiKa18} and Hieber \& Pr\"uss \cite{HiPr17,HiPr18} for an associated $L^q$-theory. The purpose of this article is to show that the system \eqref{b1}--\eqref{b6} actually fits into the classical framework of gradient systems as developed for example in Brezis \cite{Br73}, if one uses the intermediate language of so-called $j$-subgradients developed in Chill, Hauer \& Kennedy \cite{ChHaKe16}; see also \cite{BeCh15} by the authors for an application of this theory. Actually, the $j$-subgradient {\em is} a classical subgradient. Existence and uniqueness of solutions, that is, generation of a nonlinear semigroup, thus follows from classical results. Due to the special structure in the present situation, we not only obtain an associated semigroup on $L^2 (\Omega )\times L^2 (\Omega )$ but also strong solutions with values in the underlying energy space. % but on the whole scale of $L^r (\Omega)$-spaces, $r\in [1,\infty [$ (true??) and that this semigroup is order preserving (true????? perhaps not). \\

\section{Wellposedness in {$L^2 (\Omega ) \times L^2 (\Omega )$} and in the energy space}

In order to formulate the system as an abstract gradient system we consider the energy space 
\begin{align*}
 V := \{ (u_i ,u_e, w) \in W^{1,p} (\Omega ) \times W^{1,p} (\hat{\Omega} ) \times & L^2 (\Omega ) : \int_{\Omega} u_e = 0 \\
 & \text{ and } u_i -u_e|_\Omega \in L^2 (\Omega ) \} 
\end{align*}
equipped with a canonical norm, so that $V$ becomes a (reflexive) Banach space, and the energy functional $\E : V \to \R\cup\{+\infty \}$ given by
\begin{equation} \label{eq.energy}
 \E (u_i,u_e,w) = \int_\Omega Q_i (x,\nabla u_i) + \int_{\hat{\Omega}} Q_e (x,\nabla u_e) + \int_\Omega F(u_i-u_e,w) .
\end{equation}
Note that the functions under the integrals are bounded from below by $0$ (first two integrals) and by a quadratic function due to the semiconvexity of $F$ (third integral), but that we have not imposed any growth conditions from above; in particular, the three integrals only exist in $\R\cup \{+\infty\}$. As mentioned above, it has been remarked in \cite{AmCoSa00} that the system \eqref{b1}--\eqref{b6} exhibits a gradient structure with respect to the energy functional defined in \eqref{eq.energy}. However, it seems that this observation has been exploited in order to obtain wellposedness only in the semilinear case and when $\Omega = \hat{\Omega}$. Due to the fact that the parabolic equation \eqref{b1} is coupled with the elliptic equation \eqref{b2} it is natural that one does not obtain a semigroup on the Hilbert space $L^2 (\Omega ) \times L^2 (\hat{\Omega})\times L^2 (\Omega )$, say, into which the energy space $V$ embeds continuously and injectively. In fact, the dynamics takes place only in a proper subspace of this product space. On the other hand, it has been recently observed for linear operators associated with bilinear forms \cite{ArEl12,ArEl12a} and then for abstract subgradients, that it is not necessary to embed the energy space continuously and injectively into a Hilbert space where the dynamics takes place. In the case of the bidomain problem, we consider the mapping
\begin{equation} \label{eq.j}
\begin{split}
j : V & \to L^2 (\Omega ) \times L^2 (\Omega ), \\
     (u_i,u_e,w) & \mapsto ( u_i - u_e|_\Omega ,w), 
\end{split}
\end{equation}
which is obviously linear and bounded. It has dense range, but it is not injective. \\

Let us recall some basic facts for $j$-subgradients. Given an energy function $\E : V\to\R\cup\{ +\infty\}$ defined on a Banach space $V$, given a Hilbert space $H$ with inner product $\langle \cdot,\cdot \rangle_H$,  and given a linear, bounded mapping $j:V\to H$, the {\em $j$-subgradient} of $\E$ is in \cite{ChHaKe16} defined by
\begin{align*}
\partial_j\E & := \{ (v,f) \in H\times H : \exists \hat{u}\in \dom{\E} \text{ s.t. } v = j(\hat{u}) \text{ and } \\
& \phantom{:= \{ (w,f) \in L^2(\Gamma )}  \forall \varphi \in V \, : \, \liminf_{\lambda\searrow 0} \frac{\E (\hat{u}+\lambda\varphi ) - \E (\hat{u})}{\lambda} \geq \langle f , j(\varphi ) \rangle_H  \} .
\end{align*}
Here, $\dom{\E} = \{ \E <\infty\}$ is the {\em effective domain} of the energy function $\E$. Assume that $\E$ is lower semicontinuous and {\em $j$-elliptic}. The latter means that for some $\omega\in\R$ the function $\E_\omega (u) := \E (u ) + \frac{\omega}{2}\, \|j(u)\|_H^2$ is convex and coercive, coercivity meaning in turn that the sublevel sets $\{\E_\omega\leq c\}$ are relatively weakly compact. Then $\partial_j\E$ is (up to adding a multiple of the identity) a maximal monotone operator on the Hilbert space $H$ \cite[Theorem 2.6]{ChHaKe16}. Even more is true: by \cite[Corollary 2.7]{ChHaKe16}, there exists a semiconvex, lower semicontinuous energy function $\E^H : H\to\R\cup\{ +\infty\}$ such that $\partial\E^H = \partial_j\E$, that is, the $j$-subgradient is a classical subgradient for some energy function defined on $H$. By \cite[Corollary 2.10]{ChHaKe16}, $\dom{\E^H} = j (\dom{E})$. Hence, by \cite[Th\'eor\`eme 3.6, p.72]{Br73}, %\cite[Exemple 2.3.4 p.25, Th\'eor\`emes 3.5, 3.6 on p.66, 72]{Br73}, 
the gradient system
\begin{equation} \label{eq.subgrad}
\dot v + \partial_j \E (v) \ni 0 
\end{equation}
is wellposed in the sense that $-\partial_j\E$ generates a strongly continuous semigroup $S = (S_t)_{t\geq 0}$ of Lipschitz continuous operators on $H$ (strong continuity for $t>0$, and in $t=0$ only for initial values in $\overline{j(\dom{\E})}$). For every $v_0\in H$ the orbit $v=S(\cdot )v_0$ is a strong solution of \eqref{eq.subgrad}, that is, $v\in W^{1,\infty}_{loc} ((0,\infty );H)$, $v(t)\in\dom{\partial_j\E}$ for almost every $t\in (0,\infty )$, and the inclusion \eqref{eq.subgrad} is satisfied for almost every $t$. Moreover, if $v_0\in\overline{j(\dom{\E})}$, then $v$ is continuous on $[0,\infty )$ and $v(0) = v_0$. \\ % By \cite[Proposition ??]{Br73},  

In order to show applicability of the abstract theory, we first show the following 

\begin{lemma} \label{l}
Let $\E$ be the functional defined in \eqref{eq.energy}, and $j$ the mapping defined in \eqref{eq.j}. Then $\E$ is lower semicontinuous and $j$-elliptic. The effective domain of $\E$ is given by
\begin{align*}
 \dom{\E} = \{ (u_i,u_e,v)\in V : \int_{\Omega_i} & Q_i (x,\nabla u_i) , \, \int_{\Omega_e} Q_e (x,\nabla u_e) , \\
  & \int_{\Omega_i} F(u_i-u_e,w) <+\infty \} .
\end{align*}
The range $j(\dom{\E})$ is dense in $L^2 (\Omega )\times L^2 (\Omega )$. % in the sense that the function $\E_\omega (u) = \E (u) + \frac{\omega}{2} \| u_i-u_e\|_{L^2}^2$. 
\end{lemma}

\begin{proof}
In order to show that $\E$ is lower semicontinuous and $j$-elliptic, we first choose %in the following first $\eta >0$ and then 
$\omega\in\R$ large enough so that the function $(u,w)\mapsto F(u,w) + \frac{\omega}{2} (u^2 +w^2)$ is convex (assumption \eqref{f}). Replacing then $\omega$ by $\omega +1$, if necessary, we may without loss of generality assume that this function is strictly convex and bounded from below by the function $(u,w)\mapsto \frac{1}{2} (u^2 +w^2) + d$ for some $d\in\R$. Changing the function $F$ by an additive constant, we may without loss of generality assume that $d=0$. This does not affect lower semicontinuity or $j$-ellipticity of $\E$.

In order to show that $\E$ is lower semicontinuous, let $((u_i^n,u_e^n,w^n))_n$ be a sequence in $V$ which converges (in $V$) to some element $(u_i,u_e,w)$. After passing to a subsequence, if necessary, we may without loss of generality assume that $\liminf_{n\to\infty} \E (u_i^n,u_e^n,w) = \lim_{n\to\infty} \E (u_i^n,u_e^n,w)$. After passing to a second subsequence, we may further assume that the sequences $(\nabla u_i^n)_n$, $(\nabla u_e^n)_n$, $(u_i^n-u_e^n)_n$ and $(w^n)_n$ converge almost everywhere on $\Omega$ resp. $\hat{\Omega}$. The inequality 
\[
 \E_\omega (u_i,u_e,w) \leq \lim_{n\to\infty} \E_\omega (u_i^n,u_e^n,w)  
\]
then follows from the continuity and positivity of $Q_i$, $Q_e$ (assumptions \eqref{q1} and \eqref{q3}) and $(u,w)\mapsto F(u,w) + \frac{\omega}{2} (u^2+w^2)$ (assumption \eqref{f} and the choice of the constants $\omega$ and $d$ above), and from Fatou's lemma. Since $j$ is continuous, this implies that $\E$ is lower semicontinuous.

In order to show that $\E$ is $j$-elliptic, we have to show that $\E_\omega$ is convex and coercive. The convexity of $\E_\omega$ follows from the convexity of the function $(u,w)\mapsto F(u,w) + \frac{\omega}{2} (u^2+w^2)$ and the convexity of the functions $Q_i$ and $Q_e$ (assumption \eqref{q1}). It remains to show that, for every $c\in\R$, the sublevel set $\{ \E_\omega \leq c\}$ is relatively weakly compact. Since $V$ is reflexive, it suffices to show that the sublevel sets are bounded. Fix $c\in\R$. Then by the choice of $\omega$ and $d$, and by assumption \eqref{q3}, for every $(u_i,u_e,w)\in \{\E_\omega \leq c\}$,
\begin{align*}
  0 & \leq \int_\Omega Q_i (x,\nabla u_i) \leq c , \\
 0 & \leq \int_{\hat{\Omega}} Q_e (x,\nabla u_e ) \leq c , \text{ and} \\ 
 \frac{1}{2}\, \int_\Omega \left[ (u_i-u_e)^2 +w^2 \right] & \leq \int_\Omega \left[ F(u_i-u_e,w) + \frac{\omega}{2} ((u_i-u_e)^2 +w^2) \right] \leq c . 
\end{align*}
From the third line follows $\| u_i-u_e\|_{L^2}^2 \leq 2c$ and $\| w\|_{L^2}^2 \leq 2c$ for every $(u_i,u_e,w) \in \{ \E_\omega\leq c\}$. From the second line, the assumption \eqref{q3}, the condition $\int_\Omega u_e = 0$ and the Poincar\'e-Wirtinger inequality (in slightly more general form; note that we rather assume $\int_\Omega u_e = 0$ instead of $\int_{\hat{\Omega}} u_e = 0$) we obtain, for some $\hat{\lambda} >0$ and for every $(u_i,u_e,w) \in \{ \E_\omega \leq c\}$,
\[
 \hat{\lambda} \| u_e\|_{L^p}^p \leq \| \nabla u_e \|_{L^p}^p \leq \frac{1}{\alpha} \, \int_{\hat{\Omega}} Q_e (x,\nabla u_e ) \leq \frac{c}{\alpha} ,
\]
and hence $\| u_e\|_{W^{1,p}} \leq C$ for every $(u_i,u_e,w)\in \{ \E_\omega\leq c\}$ and for some constant $C\geq 0$ depending only on $c$, $\alpha$ and $\hat{\lambda}$. We further note that for every $(u_i,u_e,w)\in \{ \E_\omega \leq c\}$,
\begin{align*}
 \left| \int_{\Omega} u_i \right| & = \left| \int_\Omega (u_i-u_e) \right| \\
  & \leq |\Omega |^\frac12 \, \| u_i-u_e \|_{L^2} \\
  & \leq |\Omega |^\frac12 \, \frac{2c}{\eta} . 
\end{align*}
As a consequence, by the Poincar\'e-Wirtinger inequality again (now the classical one, in $\Omega$), for some $\lambda >0$ and for every $(u_i,u_e,w) \in \{ \E_\omega \leq c\}$,
\begin{align*}
 \lambda \, \| u_i\|_{L^p} & \leq \lambda \, \| u_i - \bar{u}_i \|_{L^p} + \lambda \, \| \bar{u}_i\|_{L^p} \\
  & \leq \| \nabla u_i \|_{L^p} + \lambda \, \| \bar{u}_i \|_{L^p} \\
  & \leq \left( \frac{1}{\alpha} \int_\Omega Q_i (x,\nabla u) \right)^{\frac{1}{p}} + \lambda \, \| \bar{u}_i \|_{L^p} \\
  & \leq \left( \frac{c}{\alpha}\right)^\frac{1}{p} + \lambda \, \| \bar{u}_i \|_{L^p} ,
\end{align*}
where $\bar{u}_i = \frac{1}{|\Omega |} \int_\Omega u_i$, and hence also $\| u_i\|_{W^{1,p}} \leq C$ for every $(u_i,u_e,w)\in \{ \E_\omega \leq c\}$ and for some constant $C\geq 0$ depending only on $c$, $\alpha$, $\lambda$, $p$ and $|\Omega |$. Taking the preceding estimates together, we have shown that, for every $c\in\R$, the sublevel set $\{ \E_\omega \leq c\}$ is bounded in the energy space, and hence $\E_\omega$ is coercive.

The description of the effective domain is obvious. For the final statement on the image of the effective domain of $\E$ we note that for example $C_c^1 (\Omega ) \times \{ 0\} \times C_c ( \Omega )$ is contained in the effective domain of $\E$, and hence $C_c^1 (\Omega ) \times C_c (\Omega )$ is contained in $j (\dom{\E})$. Hence, %by standard density arguments for Lebesgue spaces, 
$j(\dom{\E})$ is dense in $L^2 (\Omega )\times L^2 (\Omega )$.
\end{proof}

We may thus apply the abstract theory of $j$-subgradients to the pair $(\E ,j)$ defined above. We equip the Hilbert space $H=L^2 (\Omega )\times L^2 (\Omega )$ with the slightly non-standard inner product
\[
 \langle (u,w) , (\hat{u},\hat{w} ) \rangle_\tau := \int_\Omega u\, \hat{u} + \tau\, \int_\Omega w\, \hat{w} .
\]
Moreover, we assume that $q_i$ and $q_e$ satisfy the growth conditions
\begin{equation} \label{q5}
\begin{split}
& |q_i (x,y)| + |q_e (x,y)| \leq C\, (|y|^{p-1} +1) \text{ for some } C\geq 0 \text{ and} \\
& \text{every } x\in \Omega \text{ (resp. } x\in\hat{\Omega} \text{) and every } y\in\R^N ,
\end{split}
\end{equation}
and for the partial derivatives of $F$ we assume 
\begin{equation} \label{f2}
\begin{split}
& |\frac{\partial F}{\partial u} (u,w)| \leq C \, ( \sum_{k=1}^n |u|^{\alpha_k} \, |w|^{\beta_k} +1 ) \text{ and }\\
& |\frac{\partial F}{\partial w} (u,w)| \leq C \, ( \sum_{k=1}^n |u|^{\gamma_k} \, |w|^{\delta_k} +1 )\\
& \text{ for some } C\geq 0 , \, \alpha_k , \, \beta_k , \, \gamma_k , \, \delta_k \geq 0 \text{ satisfying } \\
& \frac{\alpha_k +1}{p^*} + \frac{\beta_k}{2} \leq 1  \text{ and } \frac{\gamma_k}{p^*} + \frac{\delta_k +1}{2} \leq 1 , \\
& \text{where } p^*  \begin{cases} 
                       = \frac{Np}{N-p} & \text{if } p<N , \\
                       < +\infty & \text{if } p=N , \\
                       = +\infty & \text{if } p>N .
                      \end{cases}
\end{split}
\end{equation}
Under these additional assumptions \eqref{q5} and \eqref{f2}, the energy functional $\E$ is differentiable on $V$, and hence the $j$-subgradient of $\E$ is the operator given by 
\begin{align*}
 \partial_j\E & = \{ (u,w,f,g)\in L^2 (\Omega )^4 : \exists (u_i,u_e)\in W^{1,p} (\Omega ) \times W^{1,p} (\hat{\Omega} ) \text{ s.t. }  \\ 
 & \phantom{:= \{ (v,f)\in  :} (u_i,u_e,w)\in\dom{\E} , \, u = u_i - u_e , % \, q_i (x,\nabla u_i) \in L^{p'} (\Omega ), \\
% & \phantom{:= \{ (v,f)\in  :} q_e (x,\nabla u_e)\in L^{p'} (\hat{\Omega} ), \,  \frac{\partial F}{\partial u} (u_i-u_e,w) \in L^2 (\Omega ) , \\
% & \phantom{:= \{ (v,f)\in  :} \frac{\partial F}{\partial w} (u_i-u_e,w)\in L^2 (\Omega ) 
\text{ and } \forall (\varphi , \psi ,\chi ) \in V \\
 & \phantom{:= \{ (v,f)\in  :} \int_\Omega q_i (x,\nabla u_i) \nabla \varphi + \int_{\hat{\Omega}} q_e (x,\nabla u_e)\nabla\psi + \\
 & \phantom{:= \{ (v,f)\in  :} + \int_\Omega \frac{\partial F}{\partial u} (u_i-u_e,w) (\varphi - \psi) + \int_\Omega \frac{\partial F}{\partial w} (u_i-u_e,w) \chi = \\
 & \phantom{:= \{ (v,f)\in  :} = \int_\Omega f \, (\varphi - \psi ) + \tau\, \int_\Omega g\, \chi \} .
\end{align*}
As a consequence, $(u,w,f,g)\in\partial_j\E$ if and only if $u = u_i - u_e$ for some $u_i$, $u_e$ with $(u_i,u_e,w)\in V$ such that, in a weak sense,
%\begin{equation} \label{eq.stationary}
%\begin{split}
\begin{align}
\label{s1} - \dv q_i (x,  \nabla u_i ) + \frac{\partial F}{\partial u} (u_i-u_e,w) & = f && \text{in } \Omega ,\\
\label{s2} \dv q_i (x, \nabla u_i ) + \dv q_e (x, \nabla u_e ) & = 0 && \text{in } \hat{\Omega} , \\
% \dv q_e (x, \nabla u_e ) & = 0 && \text{in } \hat{\Omega}\setminus \Omega , \\
% u_i - u_e & = v  && \text{in } \Omega , \\
\label{s3} q_i (x, \nabla u_i) \cdot n & = 0 && \text{on } \partial\Omega , \\
\label{s4}  q_e (x, \nabla u_e) \cdot n & = 0 && \text{on } \partial\hat{\Omega} , \\
\label{s5} \frac{1}{\tau}\, \frac{\partial F}{\partial w} (u_i-u_e,w) & = g && \text{in } \Omega . 
% q_e (x, \nabla u_e) \cdot n & = 0 && \text{on } \partial\hat{\Omega} \setminus \partial\Omega .
\end{align}
%\end{split}
%\end{equation}
We thus recognize the system \eqref{b1}--\eqref{b6} as a gradient system associated with the $j$-subgradient of $\E$. Accordingly, we call a function $(u_i,u_e,w)\in L^\infty_{loc} ((0,\infty );V)$ a {\em strong solution} of the system \eqref{b1}--\eqref{b6} if for $u := u_i-u_e$ one has $(u,w) \in W^{1,\infty}_{loc} ( (0,\infty );L^2 (\Omega )\times L^2 (\Omega )) \cap C([0,\infty );L^2 (\Omega )\times L^2 (\Omega ))$, if $u(0) = u_0$ and $w(0)=w_0$, and if for almost every $t$ the triple $(u_i(t),u_e(t),w(t))$ is a solution of the system \eqref{s1}--\eqref{s5} with $f=-\partial_t u(t)$ and $g=-\partial_t w (t)$.

\begin{theorem}[Wellposedness in $L^2 (\Omega )\times L^2 (\Omega )$ and in the energy space]
For every $(u_0,w_0)\in L^2 (\Omega )\times L^2 (\Omega )$ the bidomain system \eqref{b1}--\eqref{b6} admits a unique strong solution $(u_i ,u_e ,w) \in L^\infty_{loc} ( (0,\infty ) ;V)$. %satisfying $u_i - u_e|_\Omega =: v\in C([0,\infty ) , L^2 (\Omega )) \cap W^{1,\infty}_{loc} ((0,\infty ) ; L^2 (\Omega ))$. 
Moreover, the function $\E (u_i,u_e,w)$ is decreasing and locally absolutely continuous on $(0,\infty )$. The mapping $S_t(u_0,w_0) := (u_i(t)-u_e(t),w(t))$ defines a nonlinear, strongly continuous semigroup $S = (S_t)_{t\geq 0}$ of Lipschitz continuous mappings on $L^2(\Omega ) \times L^2 (\Omega )$ with $\|S_t\|_{Lip} \leq e^{\omega t}$, where $\omega$ is chosen as in assumption \eqref{f}. 
\end{theorem}

\begin{proof}
Let the energy space $V$, the energy functional $\E$, the mapping $j$ and the $j$-subgradient $\partial_j\E$ be defined as above. By Lemma \ref{l}, $\E$ is lower semicontinuous and $j$-elliptic. Moreover, $j(\dom{\E})$ is dense in $L^2 (\Omega )\times L^2 (\Omega )$. As we recalled above, there exists a semiconvex, lower semicontinuous energy $\E^{L^2\times L^2} : L^2 (\Omega )\times L^2 (\Omega ) \to \R\cup\{ +\infty\}$ with dense effective domain such that $\partial_j\E = \partial \E^{L^2\times L^2}$. %Since $\dom{\E^{L^2\times L^2}} = j(\dom{\E})$, and by Lemma \ref{l}, the functional $\E^{L^2\times L^2}$ has dense effective domain.  
Thus, $-\partial_j\E$ generates a nonlinear, strongly continuous semigroup $S = (S_t)_{t\geq 0}$ of Lipschitz continuous mappings on $L^2 (\Omega ) \times L^2 (\Omega )$. Since $\E_\omega$ is convex when $\omega$ is chosen as in assumption \eqref{f}, then $\| S_t\|_{Lip}\leq e^{\omega t}$. Moreover, for every $(u_0,w_0) \in L^2 (\Omega )\times L^2 (\Omega )$ the function $(u,w) = S_\cdot (u_0,w_0)\in C([0,\infty ) , L^2 (\Omega )\times L^2 (\Omega )) \cap W^{1,\infty}_{loc} ((0,\infty ) ; L^2 (\Omega )\times L^2(\Omega ))$ is the unique strong solution of the abstract gradient system \eqref{eq.subgrad}. It satisfies $(u(t),w(t))\in\dom{\partial_j\E}$ for almost every $t\in (0,\infty )$, and $u(0) = u_0$, $w(0)=w_0$. Finally, $\E^{L^2\times L^2} (u,w)$ is a decreasing, locally absolutely continuous function on $(0,\infty )$. 

Using the definition of the $j$-subgradient, we see that for almost every $t\in (0,\infty )$ there exists a pair $(u_i(t),u_e(t))\in W^{1,p} (\Omega ) \times W^{1,p} (\hat{\Omega} )$ such that $(u_i(t),u_e(t),w(t))\in V$ and $u_i (t) - u_e (t) = u(t)$, and such that $(u_i,u_e,w)$ is a weak solution of the system \eqref{s1}--\eqref{s5} with $f = -\partial_t u (t)$ and $g=-\partial_t w(t)$. Actually, the element $(u_i(t),u_e(t),w(t))$ is a minimizer of $\E$ on the set $\{ u_i-u_e = -\partial_t u(t)\}$.  Using the strict convexity of $Q_i$ and $Q_e$ (assumption \eqref{q1}), the latter description implies that the pair $(u_i(t),u_e(t))$ is uniquely determined. Standard arguments on $\Gamma$-convergence, using the lower semicontinuity of $\E$, imply that the function $(u_i,u_e,w)$ is measurable with values in $V$. 

Recall that $\E^{L^2\times L^2} (u,w) = \E (u_i,u_e,w)$ for almost every $t\in (0,\infty )$. This equality implies that the function $\E (u_i,u_e,w)$ is decreasing and locally absolutely continuous on $(0,\infty )$. In particular, $\E (u_i,u_e,w)$ is locally bounded. Since $(u,w)$ is locally bounded with values in $L^2 (\Omega )\times L^2 (\Omega )$, we deduce that $\E_\omega (u_i,u_e,w)$ is locally bounded, where $\omega$ is chosen as in the proof of Lemma \ref{l}. Since $\E_\omega$ is coercive, this implies that $(u_i,u_e,w)$ is locally bounded with values in $V$. The function $(u_i,u_e,w)$ is thus unique strong solution of the bidomain system \eqref{b1}--\eqref{b6}. 
\end{proof}

\section{The semilinear case and first numerical tests} 

Given $M_i\in L^\infty (\Omega , \R^{N\times N} )$, $M_e \in L^\infty (\hat{\Omega} , \R^{N\times N})$ uniformly elliptic and symmetric coefficients, the semilinear system %$F(s) + \frac{\omega}{2} s^2$
\begin{align}
\label{l1} \partial_t (u_i-u_e) - \dv (M_i \nabla u_i ) + \frac{\partial F}{\partial u} (u_i-u_e,w) & = 0 && \text{in } (0,\infty )\times \Omega ,\\
 \dv (M_i \nabla u_i ) + \dv (M_e \nabla u_e ) & = 0 && \text{in } (0,\infty )\times \hat{\Omega} , \\
% \dv (M_e \nabla u_e ) & = 0 && \text{in } \hat{\Omega}\setminus \Omega , \\
% u_i - u_e & = v  && \text{in } \Omega , \\
 (M_i \nabla u_i) \cdot n & = 0 && \text{on } (0,\infty )\times \partial\Omega , \\
% (M_e \nabla u_e) \cdot n & = 0 && \text{on } \partial\Omega , \\
\label{l4} (M_e \nabla u_e) \cdot n & = 0 && \text{on } (0,\infty )\times \partial\hat{\Omega} , \\
\label{l5} \tau\, \partial_t w + \frac{\partial F}{\partial w} (u_i-u_e,w) & = 0 && \text{in } \Omega , \\
\label{l6} u_i (0,\cdot ) - u_e (0,\cdot ) & = u_0  && \text{in } \Omega , \\
\label{l7} w(0) & = w_0 && \text{in } \Omega . 
\end{align}
is a special case of the system \eqref{b1}--\eqref{b6}. This system of PDEs may be considered as a (very) simplified model describing the electrical activity of a neural tissue and the propagation of the electrical signals within this tissue. In fact, brain activity is the result of complex electro-chemical reactions resulting in the creation of an electric field propagating in all areas of the brain, as well as in the cranium. This electric field, called electroencephalogram (EEG), is measured by placing electrodes at specific locations of the skull. %The EEG is used as a potential biomarker for the diagnosis of some central nervous system diseases (CNS), such as epilepsy \cite{Sc69,Bz15}. 
The model is also widely used for simulating the electrical activity in the heart (electrocardiogram, ECG). %,  in the heart electrical activity but its use in measuring the electrical neural activity, which is more complex, is at the best, a rough macroscopic description of such an activity. 
Based on a neuron model and appropriate conductivity tensors, the propagation of the electrical signal in a neural tissue (or in the heart) may be derived from the Hodgkin-Huxley model \cite{HoHu52} resulting, after some simplifications and a homogenization process, in the bidomain system given above, or actually a linear perturbation of this system in the ordinary differential equation \eqref{l5}.

Similar to the cardiac tissue, the neuronal tissue, defined by the domain $\Omega$, can be modeled by decomposing it into three distinct regions: the cells forming the intracellular domain, the extracellular domain representing the outside of the cells, and the cellular membrane separating them. Each zone has an intracellular, an extracellular and a membrane potential, namely the functions $u_i$, $u_e$ and the voltage across the membrane, which is defined by the difference between them. The electrical activity in the skull (where the electrodes are located) is given by the equations in $\widehat{\Omega}\setminus\Omega$. There exist various models for the semilinear terms in equations \eqref{l1} and \eqref{l5} which describe the ionic currents and represent a simplification of the Hodgkin-Huxley system of equations. A well known model is the FitzHugh-Nagumo model (see FitzHugh \cite{Fi55,Fi61} and Nagumo, Arimoto \& Yoshizawa \cite{NaArYo62}), in which the equations \eqref{l1} and \eqref{l5} are replaced by (note that $u=u_i-u_e$)
\begin{align}
\label{l1s} \tag{\ref{l1}'} \partial_t u - \dv (M_i \nabla u_i ) + u(u-a)(u-1)+w & = 0 && \text{in } (0,\infty )\times \Omega ,\\
\label{l5s} \tag{\ref{l5}'} \tau\, \partial_t w - u + \lambda +\mu w & = 0 && \text{in } \Omega ,
\end{align}
for some constants $a\in [0,1]$, $\lambda$, $\mu\in\R$. When we choose a double well potential $G$ with $G'(u) = u(u-a)(u-1)$ and when we set
\[
 F(u,w) = G(u) +uw +\lambda w + \frac{\mu}{2} w^2 ,
\]
so that 
\begin{align*}
\frac{\partial F}{\partial u} (u,w) & = u(u-a)(u-1)+w \quad \text{ and} \\
\frac{\partial F}{\partial w} (u,w) & = u + \lambda +\mu w ,
\end{align*}
then equations \eqref{l1} and \eqref{l1s} coincide, but the ordinary differential equation \eqref{l5} slightly differs from equation \eqref{l5s} due to the different sign in front of $u$. We proceed nevertheless as in the previous section, adding however a linear perturbation in the ordinary differential equation. The energy space is now
\[
 V = \{ (u_i,u_e,w)\in H^1 (\Omega ) \times H^1 (\hat{\Omega} ) \times L^2 (\Omega ) : \int_\Omega u_e = 0\} ,
\]
and the energy $\E : V \to \R\cup\{+\infty\}$ is given by 
\[
\E (u_i,u_e,w) = \frac12 \int_\Omega (M_i \nabla u_i ) \, \nabla u_i + \frac12 \int_{\hat{\Omega}} (M_e \nabla u_e ) \, \nabla u_e + \int_\Omega F(u_i-u_e,w) . 
\]
For simplicity, we restrict ourselves to dimension $N=2$ and hence, since the growth conditions \eqref{q5} and \eqref{f2} are satisfied, and as a consequence of the Sobolev embedding theorem, the energy is continuously differentiable on the energy space. Thanks to the $j$-gradient structure we resort to a descent method to solve the bidomain problem. Given a partition $\sigma : 0=t_0 <t_1 <\dots < t_n =T$ of a bounded intervall $[0,T]$, and given initial values $u_0$, $w_0\in L^2(\Omega)$, we are seeking piecewise affine functions $u$, $w\in C([0,T];L^2(\Omega ))$ (affine on the intervalls $[t_k,t_{k+1}]$) solving the implicit Euler scheme
\begin{equation} \label{eq1}
\begin{split}
& \frac{(u(t_{k+1}),w(t_{k+1})) - (u(t_k) ,w(t_k))}{t_{k+1}-t_k} + \d_j\E (u(t_{k+1}),w(t_{k+1})) \ni (0,\frac{2}{\tau} \, u(t_{k+1}) ) \\ %\quad (0\leq k\leq n-1), \\
& u(0) = u_0, \quad w(0) = w_0 ,
\end{split}
\end{equation}
where the right-hand side is a correction term arising from the Fitzhugh-Nagumo model (note that we replace the ordinary differential equation \eqref{l5} by \eqref{l5s}). Using the definition of the $j$-subgradient and writing as before $u = u_i-u_e$, this implicit Euler scheme leads to the following variational formulation, in which all functions $u_i$, $u_e$ and $w$ are assumed to take values in $L^2 (\hat{\Omega} )$ and in which we added regularization terms:  
\begin{align*}
& \int_\Omega \frac{u(t_{k+1})-u(t_k)}{t_{k+1}-t_k} \, \varphi + \int_\Omega M_i \nabla u_i (t_{k+1}) \, \nabla\varphi + \int_\Omega (G ( u (t_{k^*})) + w (t_{k+1}) ) \,\varphi  && \\
& \quad + \tau\, \int_\Omega \frac{w (t_{k+1}) - w (t_k)}{t_{k+1}-t_k} \,\chi + \int_\Omega (-u (t_{k+1}) + \lambda + \mu w (t_{k+1})) \,\chi && \\ 
& \quad - \int_\Omega M_i \nabla u_i (t_{k+1}) \, \nabla\psi + \int_\Omega M_e \nabla u_e (t_{k+1}) \,\nabla\psi +\ve \int_\Omega u_e (t_{k+1}) \, \psi && \\
& \quad + \int_{\hat{\Omega}\setminus\Omega} M_e \nabla u_e (t_{k+1})\, \nabla\psi % && \\%+ \ve \, \int_{\hat{\Omega}\setminus\Omega} u_e (t_{k+1})\,\psi && \\ 
%& \quad + 
+ { \ve \int_{\hat{\Omega}\setminus\Omega} u_i (t_{k+1})\, \phi } + { \ve \, \int_{\hat{\Omega}\setminus\Omega} w (t_{k+1}) \,\chi} && \\
%\quad + (\frac{w^{n+1}-w^n}{k},\eta)-\tau^{-1}(u^{n+1}-\lambda-\mu w^{n+1},\eta)_{\O_1}+\textcolor{red}{\ve (w^{n+1},\eta)_{\O_2} }&&\\ 
&\quad =0 \qquad \text{for every } (\varphi, \psi ,\chi )\in H^1 (\widehat\O)^3 . &&
\end{align*}
Here, in the first line, either $k^* = k$ or $k^* = k+1$, depending on whether the nonlinear term is treated explicitly or in an implicit way using any standard scheme (for example, Newton's method). The algorithm is respectful of the physiological transmission and boundary conditions of the model. We emphasize that the abstract theory of the $j$-subgradient ensures the existence of (only) an implicit energy on $L^2 (\Omega )\times L^2 (\Omega )$. However, our algorithm uses exactly the $j$-subgradient associated with the system \eqref{l1}--\eqref{l4}, \eqref{l5s}, \eqref{l6} and \eqref{l7}. The convergence of the algorithm, the stability and other issues will not be considered in this article, however, notice that we have a gradient structure for the system which allows us to consider such questions in the framework of the standard numerical analysis of gradient theory. As a proof of the concept, we present two numerical examples. 

\begin{figure}[h]
          {\epsfig{
           %file=../manif.eps ,
           file=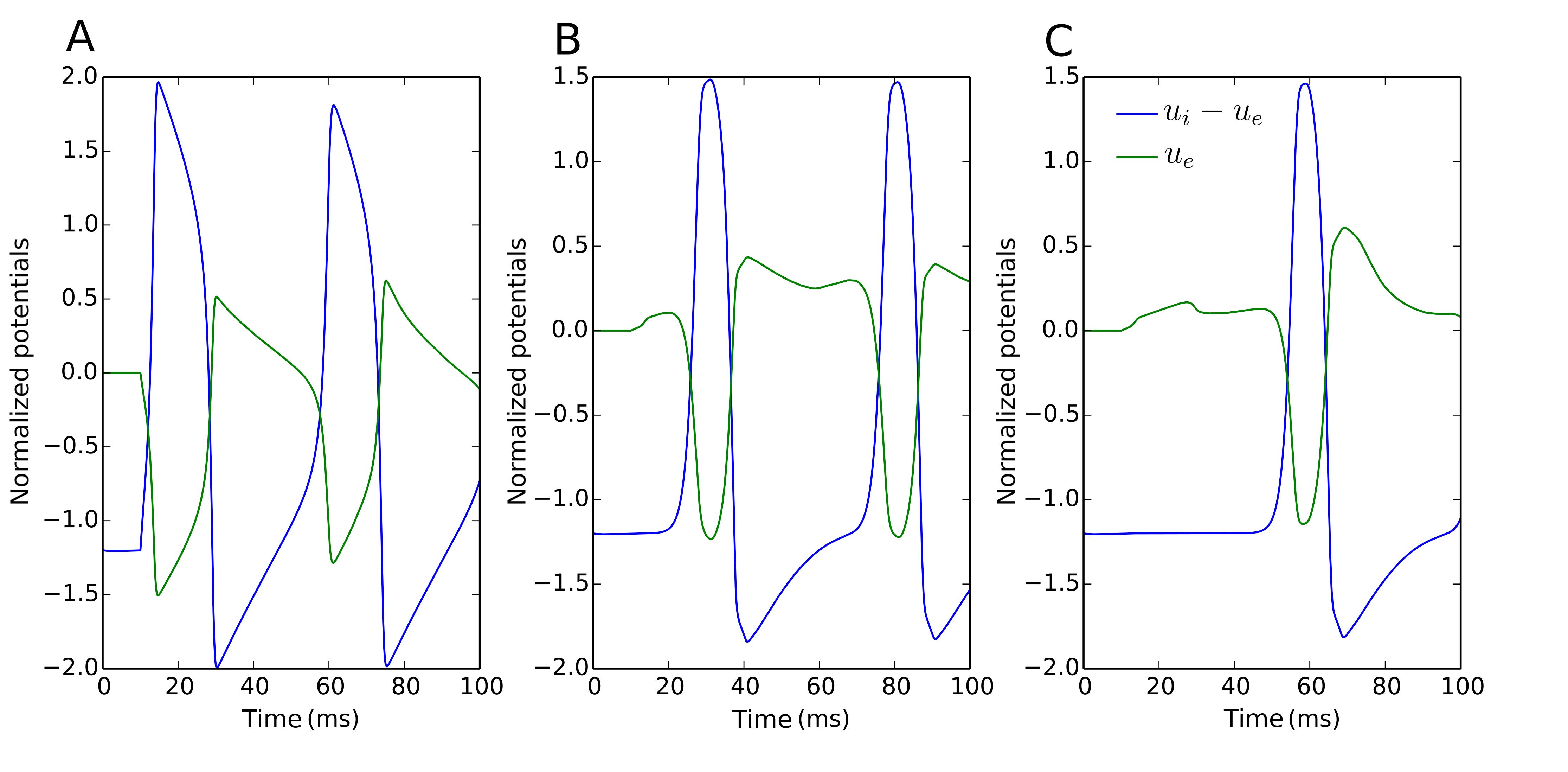 ,
           angle=0, width=120mm}
          }
\caption{Evolution of $u = u_i-u_e$ and $u_e$ at the points $(0.5,0.5)$, $(0.5,0.7)$ and $(0.5,0.9)$ in $\Omega$}
\end{figure}

In the first example, $\Omega$ is the unit disk centered at the origin. In Figure 1 we show, the action potential as a function of time evaluated at three particular points in $\Omega$ and in the case of the absence of skull, that is, $\widehat\Omega=\Omega$.  The values of the various parameters are taken from Bedez \cite{Bz15}. Namely, we apply a current $I_{\text{app}}=0.4\mu A$ on a disk of radius $0.1$ and we take $M_i=0.638 \, Id$, $M_e=1.538 \, Id$ (where $Id$ is the identity matrix).
We note that even if the bidomain system of equations may appear to be too rough and less realistic for modelling  the electrical neural activity (contrary to the case of heart activity where the cells are all similar and complex transmission processes are absent), the result obtained looks, at a macroscopic scale of a tissue of neurons, similar  to several results in the literature obtained with other models derived from the Hodgkin-Huxley theory for biological cells; see, for example, Sadleir \cite{Sa10} and Bedez \cite{Bz15} in the case of the neurons or Coudi\`ere, Pierre, Rousseau \& Turpault \cite{CPRT09} and Colli-Franzone, Pavarino \& Savar\'e \cite{CoPaSa06} in the case of the electrical heart activity. 

\begin{figure}[t]
  \centering
  \includegraphics[width=6.2cm]{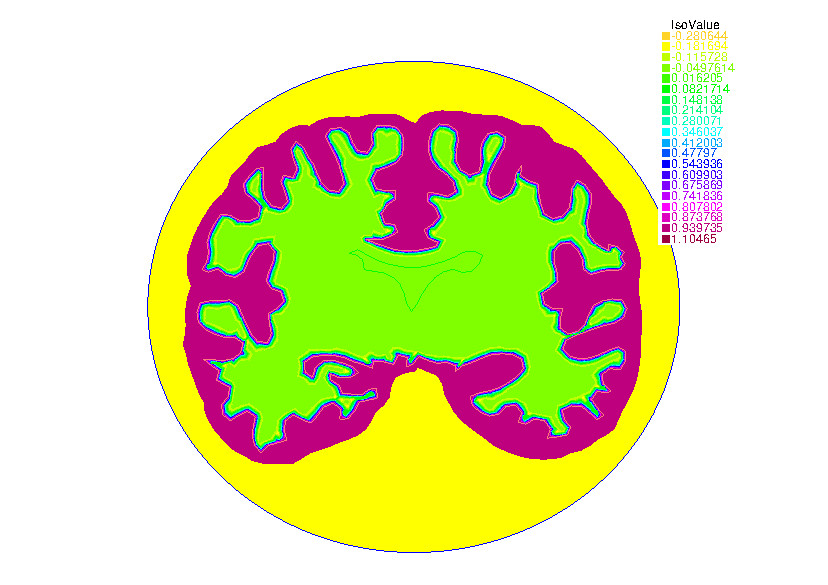}\ 
  \includegraphics[width=6.2cm]{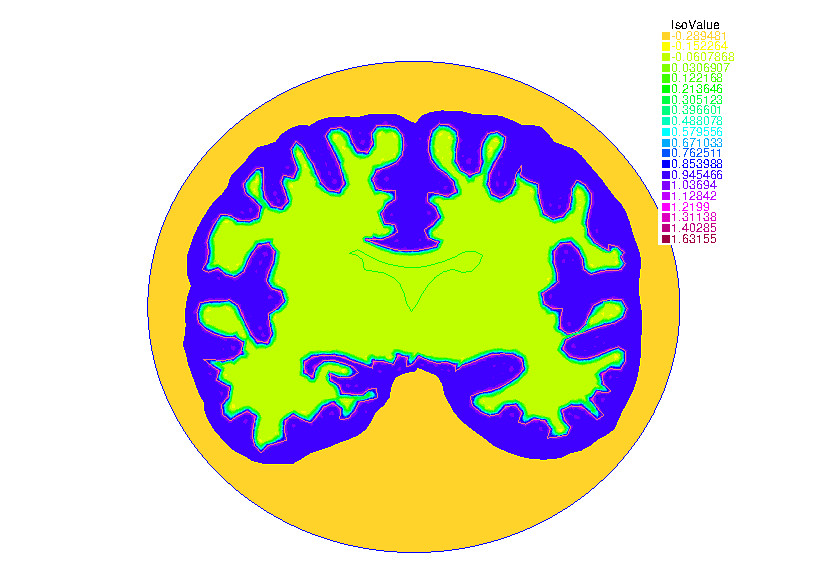} \\
    \includegraphics[width=6.2cm]{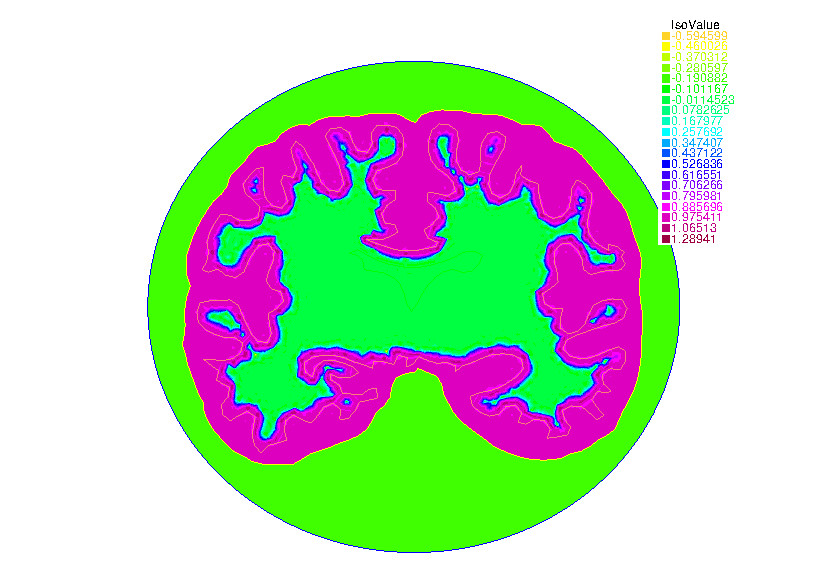}\ 
  \includegraphics[width=6.2cm]{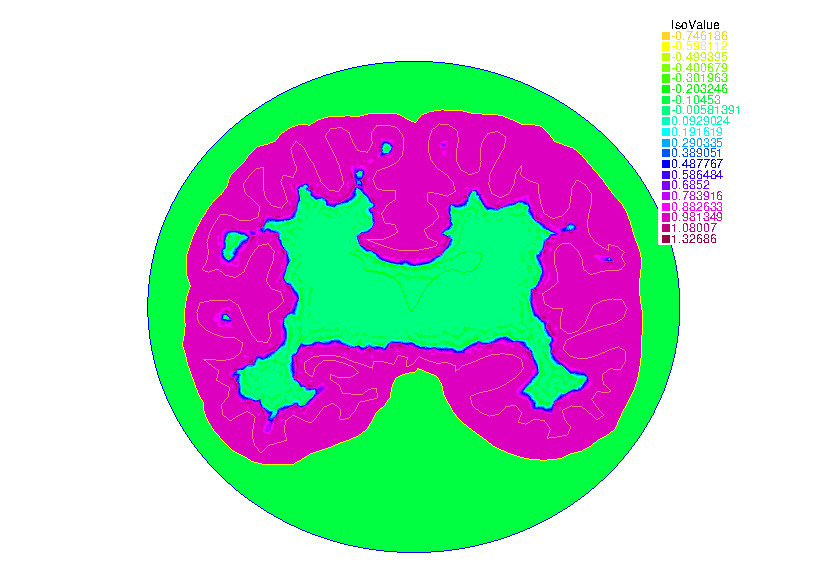} \\
    \includegraphics[width=6.2cm]{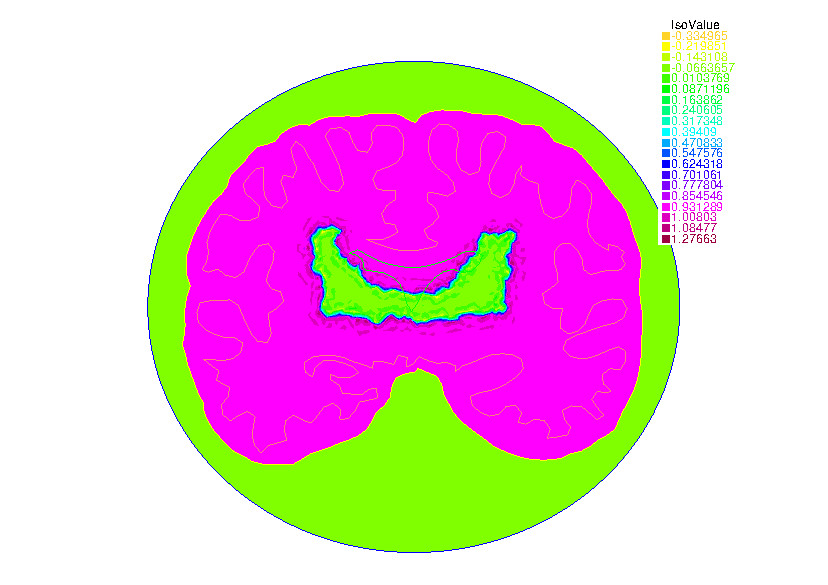}\ 
  \includegraphics[width=6.2cm]{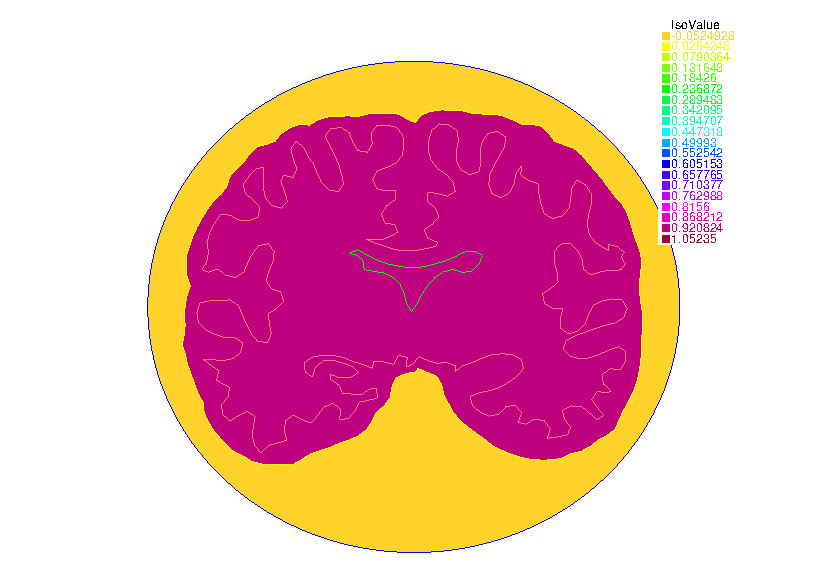} \\

%\vskip-11cm
  \caption{Evolution of $u = u_i-u_e$ in $\Omega$ and $u_e$ in $\hat{\Omega}\setminus\Omega$ at six different times}
  \label{fig:picture1012}
\end{figure}

In the second example, the underlying domains $\Omega$ and $\hat{\Omega}\setminus\Omega$ are two-dimensional models of a human brain and the skull, respectively.  In this example the conductivities are the diagonal matrices $M_i = {\rm diag}\, (0.41,0.47)$ and $M_e= {\rm diag}\, (0.29, 0.61)$ in $\Omega$, and $M_e=1.2 \, Id$ in the skull. We initialize $u=1$ in the cortex region and $u=0$ elsewhere. In Figure 2, we show the evolution of $u = u_i-u_e$ in $\Omega$ and of $u_e$ in $\hat{\Omega}\setminus\Omega$ at six different times. We do not apply any electrical stimulus and we observe the effect of anisotropy and complex geometry of the model. To verify the relevance of the model for neurons we should distinguish the conductivities in different regions of the brain (and also in grey and white matter) but this is not our purpose in this article.

We emphasize that the theoretical framework and the algorithm considered here permit to solve the bidomain problem for various settings (with or without skull, different boundary conditions and nonlinear potentials, \dots). However, a serious numerical analysis as performed in Coudi\`ere et al. \cite{CPRT09} and Colli Franzone et al. \cite{CoPaSa06} (see also the references therein) should be conducted to complete this study. In particular, the algorithm with the double well potential yields, in most cases, the convergence to a stationary solution (the nonlinearity ensures the transition from left to right potentials) but the choice of several parameters (the regularization coefficients, the step size, the mesh size, \dots) are not yet studied. Such choices are crucial to capture relevant solutions, for example, travelling waves.

The point of view of the $j$-subgradient for elliptic-parabolic systems like the bidomain model, in addition to be an elegant approach, makes both the analysis and the computations easier, as it may benefit from the standard theory and numerical tools provided in a unified gradient systems framework.

\medskip

{\bf Acknowledgment.} The authors are most grateful to Fr\'ed\'eric Hecht for providing a mesh of a human brain for the numerical test.

\nocite{AmCoSa00}
\nocite{CoPaSa06}
\nocite{PeSaCo05}
\nocite{CoSa02}
\nocite{CoPaSc14}
\nocite{NeKr93}
\nocite{Ve06,Ve09}
\nocite{BeCh15}

\providecommand{\bysame}{\leavevmode\hbox to3em{\hrulefill}\thinspace}

\bibliographystyle{amsplain}
%\bibliography{../ralph}

\def\cprime{$'$}
  \def\ocirc#1{\ifmmode\setbox0=\hbox{$#1$}\dimen0=\ht0 \advance\dimen0
  by1pt\rlap{\hbox to\wd0{\hss\raise\dimen0
  \hbox{\hskip.2em$\scriptscriptstyle\circ$}\hss}}#1\else {\accent"17 #1}\fi}
  \def\cprime{$'$} \def\cprime{$'$} \def\cprime{$'$}
\providecommand{\bysame}{\leavevmode\hbox to3em{\hrulefill}\thinspace}
\providecommand{\MR}{\relax\ifhmode\unskip\space\fi MR }
\providecommand{\MRhref}[2]{%
  \href{http://www.ams.org/mathscinet-getitem?mr=#1}{#2}
}
\providecommand{\href}[2]{#2}

\end{document}